\documentclass[12pt]{article}
\usepackage{amssymb}
\usepackage{amsmath}
\usepackage{amsfonts}

\newtheorem{theorem}{Theorem}[section]
\numberwithin{theorem}{subsection}

\newtheorem{corollary}[theorem]{Corollary}

\newtheorem{definition}[theorem]{Definition}

\newtheorem{lemma}[theorem]{Lemma}

\newtheorem{proposition}[theorem]{Proposition}
\newtheorem{remark}[theorem]{Remark}

\newenvironment{proof}[1][Proof]{\noindent \textbf{#1.} }{\  \rule{0.5em}{0.5em}}
\begin{document}

\begin{center}
\textbf{WHEN IS THE RING OF INTEGERS OF A NUMBER FIELD COVERABLE ?}

\ \ 

Mohamed Ayad, Laboratoire de Math\'ematiques Pures et Appliqu\'ees, Universit\'e du Littoral, F-62228 Calais. France

E-mail: ayadmohamed502@yahoo.com\\

Omar Kihel, Department of Mathematics, Brock University, Ontario, Canada L2S
3A1, Canada

E-mail : okihel@brocku.ca

\ 

\ 
\end{center}

\smallskip

\smallskip

\noindent \textbf{Abstract. }A commutative ring $R$ is said to be coverable if it is
the union of its proper subrings and said to be finitely coverable if it is the union
of a finite number of them. In the latter case, we denote by $\sigma (R)$ the
minimal number of required subrings.

In this paper, we give necessary and sufficient conditions for the ring of integers $A
$ of a given number field to be finitely coverable and a formula 
for $\sigma (A)$ is given which holds when they are met. The conditions are expressed in terms of the
existence of common index divisors and (or) common divisors of values of
polynomials.

\section{Introduction}
Covering a group $G$ by some of its finite subsets having a certain property was first considered by Neumann \cite{Ne} in $1954$. Many people later considered the question of identifying when $G$ is a union of some of its proper subgroups, naturally leading to the study of the covering number of $G$ -- the minimal number $n$ needed to write $G$ as the union of $n$ of its proper subgoups. Calculating covering numbers of groups is a problem not yet completely solved, although there are some families of groups for which the covering numbers are known.  For instance, Tomkinson \cite{Tm} calculated the covering number for solvable groups and Bryce, Fedri and Serena \cite{B.F.S} for some linear groups. Bhargava \cite{Ba} studied when a group is a union of some of its normal proper subgroups.
 
It is natural to consider the analogous question of covering rings. This problem is in fact a nontrivial one, as noted by Werner \cite{We} who initially solved it for finite semisimple rings.

Let $K$ be a number field of degree $n$ and $A$ its ring of integers.
Denote by $\hat{A}$ the set of elements of $A$ which are primitive over $\mathbb{Q}$.
For any $\theta \in A$, denote by $F_{\theta }(x)$ the characteristic
polynomial of $\theta $ over $\mathbb{Q}$. Let $D_{K}$ be the absolute
discriminant of $K$. It is known that for $\theta \in \hat{A}$, the
discriminant $D(\theta )$ of $\theta $ (which is equal to the discriminant
of $F_{\theta }(x)$) is given by $D(\theta )=I(\theta )^{2}D_{K}$, where $%
I(\theta )$ is the group-index of $\mathbb{Z}\left[ \theta \right] $ in $A$.

Let $I(K)=\displaystyle\gcd_{\theta\in\hat{A}}I(\theta)$. A prime number $p$
is called a common index divisor if $p$ divides $I(K)$, and any such prime satisfies
the condition $p<n$ \cite{Zy}. Hensel proved that a prime $p$ is a common
index divisor if and only if $\lambda _{p}(f)>N_{p}(f)$ for the residual
degree $f$ of some prime ideal apearing in the splitting of $p$ in $A$, where
$N_{p}(f)$ and $\lambda _{p}(f)$ respectively denote the number of monic
irreducible polynomials over $\mathbb{F}_{p}$ of degree $f$ and the number
of prime ideals $\mathcal{P}$ of $A$ lying over $p\mathbb{Z}$ with residual
degree equal to $f$.

Let $\theta \in A$ and $i(\theta )=\displaystyle\gcd_{x\in \mathbb{Z}}F_{\theta }(x)$.
H. Gunji and D. L.\ McQuillan \cite{G.M.} defined the integer $i(K)=%
\underset{\theta \in A}{\text{lcm\,}}i(\theta )$. C. R. MacCluer \cite{Ma}
showed that $i(K)>1$ if and only if there exists a prime number $p\leq n$
having at least $p$ prime ideal factors in $A$. In \cite{A.K.}, it
is proved that if $p$ is a common index divisor, then $p\mid i(K)$.
Moreover, there exists $\theta \in A$ such that $i(K)=i(\theta )$.

A commutative ring $R$ is said to be coverable if it is the union
of its proper subrings. If this can be done with only a finite number of proper
subrings, then we say that $R$ is finitely coverable and denote by $\sigma (R)$
the minimal number of proper subrings necessary to cover $R$. If $R$ is
coverable, but not by a finite number of proper subrings, we set $\sigma
(R)=\infty $. If $R$ is not coverable, we set $\sigma (R)=\hat{\infty}$,
where $\hat{\infty}$ is a new symbol satisfying the condition $m<\infty <%
\hat{\infty}$ for any positive rational integer $m$. A covering $R=\underset{%
i\in I}{\cup\,}R_{i}$ is said to be irredundant if for any $i\in
I$, $R_{i}\not\subset \underset{j\neq i}{\cup }R_{j}$. A finite covering $R=%
\underset{i=1}{\overset{m}{\cup\,}}R_{i}$ is called minimal if $m=\sigma (R)$%
. In this paper, any situation: $\sigma (R)$ finite, 
$\sigma (R)=\infty$ and $\sigma (R)=\hat{\infty}$ may occur for the ring
of integers of a number field. It is proved that the ring of integers $A$ is
finitely coverable if and only if one of the following conditions holds:%
\medskip\newline
(i) There exists a prime number $p$ which is a common index divisor.\newline
(ii) There exists a prime number $p$ such that the number of prime ideals of 
$A$ lying over $p\mathbb{Z}$ with residual degree $f$ equal to $1
$ is at least $p$\medskip

\noindent In the situation (ii), $p\mid i(K)$. Moreover, if the number of 
mentioned prime ideals is greater than $p$, then $p\mid I(K)$. Notice that
the prime numbers occuring in (i) satisfy $p<n$, whereas those
arising in (ii) fulfill $p\leq n$, where $n$ is the degree of
the number field $K$. Thus, the total number of prime numbers satisfying (i)
or (ii) is finite. The results on the covering of $A$ are based on the study
of the ideals $I$ of $A$ such that $A/I$ is coverable, and the paper ends with
the computation of $\sigma (A)$ when $A$ is finitely coverable.

\ \ 

\noindent\textbf{Notation:}\medskip

\noindent $N_{p}(f):$ The number of monic irreducible polynomials over $\mathbb{F}_{p}$
of degree $f$.\smallskip

\noindent $\lambda _{p}(f):$ The number of prime ideals of $A$ lying over $p\mathbb{Z}$%
, having their residual degree equal to $f$.\medskip

\noindent $\tau _{p}(f):$ $\left \{ 
\begin{array}{ccc}
p & \text{if} & f=1 \\ 
N_{p}(f)+1 & \text{if} & f\geq 2%
\end{array}%
\right. $\medskip

\noindent $\left \lceil z\right \rceil :$ The smallest integer greater than or equal
to $z$.\smallskip

\noindent $\upsilon (f):$ The number of distinct prime factors of $f$.\smallskip

\noindent $\left \langle \left \langle \theta \right \rangle \right \rangle :$ The
smallest subring containing $\theta $.\smallskip

\noindent $\mathbb{Z}\left[ \theta \right] :$ The smallest subring containing $1$ and $%
\theta $.\smallskip

\noindent $Irr(\theta ,F,x):$ Mimimal polynomial of $\theta $ over the field $F$.\smallskip

\noindent $I(\theta ):$ Index of $\mathbb{Z}\left[ \theta \right] $ in the ring of
integers $A$.\smallskip

\noindent $\left[ G:H\right] :$ Group index of $H$ in $G$.\smallskip

\noindent $\left( A:B\right) :$ Conductor of $B$ in the ring $A$.\smallskip

\noindent $\hat{\infty}:$ Symbol satisfying the condition $m<\infty <\hat{\infty}$ for
any integer $m$.

\section{Coverable ring of integers of a number field}
\subsection{\textbf{Preliminary results }}
\begin{theorem}\label{theorem1}
Let $K$ be a number field and $A$ be its ring of integers. Then $A$ is
coverable if and only if for every unit $\theta \in A$, we have $A\neq \mathbb{Z}%
\left[ \theta \right] $ .
\end{theorem}

For any element $a$ of a ring $R$, we denote by $\left\langle \left\langle
a\right\rangle \right\rangle $ the subring of $R$ generated by $a$. $%
\left\langle \left\langle a\right\rangle \right\rangle $ is then the set of
the elements of $R$ of the form
\begin{equation*}
\alpha _{1}a+...+\alpha _{n}a^{n}
\end{equation*}%
where $n$ is any positive integer and the $\alpha _{i}$'s are elements of $%
\mathbb{Z}$.\newline
We have not necessarily $\left\langle \left\langle a\right\rangle
\right\rangle =\mathbb{Z}\left[ a\right] $. For example, in the ring $%
\mathbb{Z}$ of the rational integers, $\left\langle \left\langle
2\right\rangle \right\rangle =2\mathbb{Z}$ which is different from $\mathbb{Z%
}\left[ 2\right] =\mathbb{Z}$.

For the proof of the above theorem, we use the following lemma.

\begin{lemma}\label{lemma2}
A ring $R$ is coverable if and only if for every $a\in R$, the subring $%
\left \langle \left \langle a\right \rangle \right \rangle $ is not equal to 
$R$.
\end{lemma}

\begin{proof}
See Lemma 1.4 \cite{We}.
\end{proof}

\begin{proof}
(Proof of the theorem \ref{theorem1}). Suppose that $A$ is coverable and that $A=\mathbb{Z}%
\left[ \theta \right] $ for some element $\theta $ of $A$. Let $A=\cup A_{i}$
be a cover of $A$ by proper subrings, then there exists $i$ such that $%
\theta \in A_{i}$. Hence $\left \langle \left \langle \theta \right \rangle
\right \rangle \subseteq A_{i}\neq A$. Since $A=\mathbb{Z}\left[ \theta %
\right] $, then $\left \langle \left \langle \theta \right \rangle
\right
\rangle \neq \mathbb{Z}\left[ \theta \right] $, hence, $\theta $ is
not a unit of $A$.\newline
Conversely, note that for every $\theta \in A$, each of the two conditions%
\newline
(i) $A\neq \mathbb{Z}\left[ \theta \right] \newline
$(ii) $A=\mathbb{Z}\left[ \theta \right] $ and $\theta $ is not a unit,%
\newline
implies that $\left \langle \left \langle \theta \right \rangle
\right
\rangle \neq A$. Hence, $A=\underset{a\in A}{\cup }\left \langle
\left
\langle a\right \rangle \right \rangle $ is a cover of $A$ by proper
subrings of $A$.
\end{proof}

The following result will be needed for the proof of Proposition \ref{proposition4}.

\begin{lemma}\label{lemma3}
Let $K$ be a number field of degree $n$ and $A$ be its ring of integers. Let 
$\theta \in A$ and $F(x)=Irr(\theta ,\mathbb{Q},x)$. Let $p$ be a prime
number and $\mathcal{P}$ be a prime ideal of $A$ lying over $p\mathbb{Z}$
with residual degree equal to $f$. Write $F(x)$ in the form $%
F(x)=P(x)^{g}\pi (x)+ph(x)$ where $P(x),\pi (x),h(x)\in \mathbb{Z}\left[ x%
\right] $ and $P(x)$ and $\pi (x)$ are monic. Suppose that $P(x)$ is
irreducible over $\mathbb{F}_{p}$, $P(\theta )\equiv 0(\mod \mathcal{P)}
$ and $P(x)$ does not divide $\pi (x)$ in $\mathbb{F}_{p}\left[ x\right] $.
Let $\mathfrak{f}=(A:\mathbb{Z}\left[ \theta \right] )$.\newline
Then $\mathcal{P}\nmid \mathfrak{f}$ if and only if $\deg P(x)=f$ and $%
\mathcal{P=}(p,\mathcal{P(\theta ))}$.
\end{lemma}

\begin{proof}
See Theorem p. 203, Art 104, Hancock. Vol. 2 \cite{Ha}.
\end{proof}

\begin{proposition}\label{proposition4}
Let $K$ be a number field, $A$ be its ring of integers, $p$ be a prime
number and $\mathcal{P}_{1},...,\mathcal{P}_{s}$ be prime ideals of $A$
lying over $p\mathbb{Z}$. Let $\theta \in A$, $I=\mathcal{P}_{1}^{h_{1}}...%
\mathcal{P}_{s}^{h_{s}}$ and $R=A/I$. Then the following conditions (1) and
(2) are equivalent.\newline
(1) $\left \langle \left \langle \overline{\theta }\right \rangle
\right
\rangle =R$.\newline
(2) (i) For any $i=1,...,s,\left \langle \left \langle \theta +\mathcal{P}%
_{i}\right \rangle \right \rangle =A/\mathcal{P}_{i}$.\newline
(ii) For any $(i,j)$ such that $i\neq j,Irr(\theta +\mathcal{P}_{i},\mathbb{F%
}_{p},x)\neq Irr(\theta +\mathcal{P}_{j},\mathbb{F}_{p},x)$.\newline
(iii) For any $i=1,...,s$ such that $\mathcal{P}_{i}^{2}\mid p$ and $%
h_{i}\geq 2,Irr(\theta +\mathcal{P}_{i},\mathbb{F}_{p},x)(\theta )\not
\equiv 0(\mod \mathcal{P}_{i}^{2})$.
\end{proposition}

\begin{proof}
(1)$\Rightarrow $(2). We prove the contrapositive.\newline
(i) Suppose that there exists $i$ such that $\left \langle \left \langle
\theta +\mathcal{P}_{i}\right \rangle \right \rangle \neq A/\mathcal{P}_{i}$%
. Let $\gamma +\mathcal{P}_{i}\in A/\mathcal{P}_{i}\setminus \left \langle
\left \langle \theta +\mathcal{P}_{i}\right \rangle \right \rangle $.
Consider the element $\eta \in R=\prod \limits_{j=1}^{s}A/\mathcal{P}%
_{j}^{h_{j}}$ such that $\eta =(0,...,0,\gamma +\mathcal{P}%
_{i}^{h_{i}},0,...,0)$. We show that $\eta \notin \left \langle
\left
\langle \overline{\theta }\right \rangle \right \rangle $. Suppose
that $\eta \in \left \langle \left \langle \overline{\theta }\right \rangle
\right
\rangle $. Then, 
\begin{equation*}
\begin{array}{ccc}
\eta & = & a_{1}(\theta +\mathcal{P}_{1}^{h_{1}},...,\theta +\mathcal{P}%
_{s}^{h_{s}})+...+a_{m}(\theta +\mathcal{P}_{1}^{h_{1}},...,\theta +\mathcal{%
P}_{s}^{h_{s}})^{m} \\ 
& = & (a_{1}\theta +...+a_{m}\theta ^{m}+\mathcal{P}_{1}^{h_{1}},...,a_{1}%
\theta +...+a_{m}\theta ^{m}+\mathcal{P}_{s}^{h_{s}})%
\end{array}%
\end{equation*}%
\newline
then $\gamma +\mathcal{P}_{i}^{h_{i}}=a_{1}\theta +...a_{m}\theta ^{m}+%
\mathcal{P}_{i}^{h_{i}}$, which shows that $\gamma +\mathcal{P}_{i}\in
\left
\langle \left \langle \theta +\mathcal{P}_{i}\right \rangle
\right
\rangle $, a contradiction.\newline
(ii) Suppose that there exists $(i,j),i\neq j$ such that :
\begin{equation*}
Irr(\theta +\mathcal{P}_{i},\mathbb{F}_{p},x)=Irr(\theta +\mathcal{P}_{j},%
\mathbb{F}_{p},x).
\end{equation*}%

Consider the element $\eta \in R$ such that $\eta =(0,...,0,1+\mathcal{P}%
_{i}^{h_{i}},0,...,0)$. We show that $\eta \notin \left \langle
\left
\langle \overline{\theta }\right \rangle \right \rangle $. Suppose
that $\eta \in \left \langle \left \langle \overline{\theta }\right \rangle
\right
\rangle $. Then
\begin{equation*}
\begin{array}{ccc}
\eta & = & a_{1}(\theta +\mathcal{P}_{1}^{h_{1}},...,\theta +\mathcal{P}%
_{s}^{h_{s}})+...+a_{m}(\theta +\mathcal{P}_{1}^{h_{1}},...,\theta +\mathcal{%
P}_{s}^{h_{s}})^{m} \\ 
& = & (a_{1}\theta +...+a_{m}\theta ^{m}+\mathcal{P}_{1}^{h_{1}},...,a_{1}%
\theta +...+a_{m}\theta ^{m}+\mathcal{P}_{s}^{h_{s}})%
\end{array}%
\end{equation*}%

Identifying the $i$th and $j$th components respectively, we obtain :
\begin{equation*}
\left \{ 
\begin{array}{ccc}
1+\mathcal{P}_{i}^{h_{i}} & = & a_{1}\theta +...+a_{m}\theta ^{m}+\mathcal{P}%
_{i}^{h_{i}} \\ 
0+\mathcal{P}_{j}^{h_{j}} & = & a_{1}\theta +...+a_{m}\theta ^{m}+\mathcal{P}%
_{j}^{h_{j}}%
\end{array}%
\right.
\end{equation*}%

We deduce that $a_{1}\theta +...a_{m}\theta ^{m}\equiv 1(\mod \mathcal{P%
}_{i})$ and $a_{1}\theta +...a_{m}\theta ^{m}\equiv 0(\mod \mathcal{P}%
_{j})$. It follows that $Irr(\theta +\mathcal{P}_{i},\mathbb{F}%
_{p},x)=Irr(\theta +\mathcal{P}_{j},\mathbb{F}_{p},x)$ divides both
polynomials $a_{1}x+...a_{m}x^{m}$ and $a_{1}x+...a_{m}x^{m}-1$ in $\mathbb{F%
}_{p}\left[ x\right] $, hence a contradiction.\newline
(iii) Suppose next that there exists $i$ such that $\mathcal{P}_{i}^{2}\mid
p $ and $h_{i}\geq 2$ and $Irr(\theta +\mathcal{P}_{i},\mathbb{F}%
_{p},x)(\theta )\equiv 0(\mod \mathcal{P}_{i}^{2})$. Consider the
element 
\begin{equation*}
\eta =(0,...,0,\pi +\mathcal{P}_{i}^{2},0,...,0)
\end{equation*}%

where $\pi $ is an element of $\mathcal{P}_{i}\setminus \mathcal{P}_{i}^{2}$%
. Suppose that $\eta \in \left \langle \left \langle \theta \right \rangle
\right \rangle $, then,
\begin{equation*}
\eta =a_{1}(\theta +\mathcal{P}_{1}^{h_{1}},...,\theta +\mathcal{P}%
_{s}^{h_{s}})+...+a_{m}(\theta +\mathcal{P}_{1}^{h_{1}},...,\theta +\mathcal{%
P}_{s}^{h_{s}})^{m}\text{,}
\end{equation*}%

hence $\pi +\mathcal{P}_{i}^{h_{i}}=a_{1}\theta +...+a_{m}\theta ^{m}+%
\mathcal{P}_{i}^{h_{i}}$. This implies that $a_{1}\theta +...+a_{m}\theta
^{m}\equiv 0(\mod \mathcal{P}_{i})$. Let $g_{i}(x)\in \mathbb{Z}\left[ x%
\right] $ such that $g_{i}(x)$ is monic irreducible over $\mathbb{F}_{p}$
and $g_{i}(\theta )\equiv 0(\mod \mathcal{P}_{i}^{2})$. Since $g_{i}(x)$
divides $a_{1}x+...+a_{m}x^{m}$ in $\mathbb{F}_{p}\left[ x\right] $, set $%
a_{1}x+...+a_{m}x^{m}=g_{i}(x)h(x)+pr(x)$, where $h(x)$ and $r(x)\in \mathbb{%
Z}\left[ x\right] $. We deduce that $a_{1}\theta +...+a_{m}\theta
^{m}=g_{i}(\theta )+pr(\theta )$. Since $\mathcal{P}_{i}^{2}\mid p$ and $%
g_{i}(\theta )\equiv 0(\mod \mathcal{P}_{i}^{2})$, then $\mathcal{P}%
_{i}^{2}\mid a_{1}\theta +...+a_{m}\theta ^{m}$, hence $\mathcal{P}^{2}\mid
\pi $, which is a contradiction.\newline
(2)$\Rightarrow $(1). Let $\mathcal{T}$ be the set of prime ideals, if any,
lying over $p\mathbb{Z}$ and not dividing $I$. We may write $pA$ in the form%

\begin{equation*}
pA=\mathcal{P}_{1}^{e_{1}}...\mathcal{P}_{s}^{e_{s}}J
\end{equation*}%

where $J$ is a product of powers of prime $\mathcal{P}_{j}\in \mathcal{T}$.
For any $i\in \left \{ 1,...,s\right \} $, let $f_{i}$ be the residual
degree of $\mathcal{P}_{i}$ over $p\mathbb{Z}$. Let $F(x)=Irr(\theta ,%
\mathbb{Q},x)$. Reducing modulo $p\mathbb{Z}$, we get
\begin{equation*}
F(x)=P_{1}(x)^{g_{1}}...P_{s}(x)^{g_{s}}H(x)+pG(x)
\end{equation*}%

where $H(x),G(x)$ and $P_{i}(x)$ are polynomials with integral coefficients
such that $P_{i}(x)$ is monic and irreducible in $\mathbb{F}_{p}\left[ x%
\right] $ for $i=1,...,s$. Moreover, we may suppose that $P_{i}(x)$ does not
divide $H(x)$ in $\mathbb{F}_{p}\left[ x\right] $ and $P_{i}(\theta )\equiv
0(\mod \mathcal{P}_{i})$ for $i=1,...,s$. Fix $i_{0}\in \left \{
1,...,s\right \} $. Since $A/\mathcal{P}_{i_{0}}$ is a field, the condition
(i) implies \ that $(\theta +\mathcal{P}_{i_{0}})^{p^{f_{i_{0}}-1}}=1+%
\mathcal{P}_{i_{0}}$, hence $\left \langle \left \langle \theta +\mathcal{P}%
_{i_{0}}\right \rangle \right \rangle =\mathbb{Z}\left[ \theta +\mathcal{P}%
_{i_{0}}\right] $. The same condition shows that $\mathbb{Z}\left[ \theta +%
\mathcal{P}_{i_{0}}\right] =A/\mathcal{P}_{i_{0}}$. It follows that $\deg
P_{i_{0}}(x)=f_{i_{0}}$, and the assertion (1) of the preceding lemma holds.%
\newline
If $\mathcal{T=\varnothing }$, that is $pA=\mathcal{P}_{1}^{e_{1}}...%
\mathcal{P}_{s}^{e_{s}}$, then obviously the assumptions (ii) and (iii) show
that $\mathcal{P}_{i}=(p,\mathcal{P}_{i}(\theta ))$ for $i=1,...,s$. In this
case the preceding lemma implies $\mathcal{P}_{i}\nmid \mathfrak{f}$, hence $%
\mathcal{P}_{i}+\mathfrak{f}=A$ for $i=1,...,s$. For any $i\in \left \{
1,...,s\right \} $, let $\alpha _{i}\in \mathcal{P}_{i}$ and $\beta _{i}\in 
\mathfrak{f}$ such that $1=\alpha _{i}+\beta _{i}$. We deduce that $1=%
\overset{s}{\underset{i=1}{\prod }}(\alpha _{i}+\beta _{i})=(\overset{s}{%
\underset{i=1}{\prod }}\alpha _{i})+\beta $, where $\beta \in \mathfrak{f}$,
hence $1=\alpha +\beta $ where $\alpha \in I$ and $\beta \in \mathfrak{f}$.
Therefore $A=I+\mathfrak{f}$, hence $A=I+\mathbb{Z}\left[ \theta \right] $.
This implies $R=\mathbb{Z}\left[ \theta +I\right] $.\newline
If $\mathcal{T\neq \varnothing }$, let $\mathcal{T=}\left \{ \mathcal{P}%
_{s+1},...,\mathcal{P}_{r}\right \} $. Here, we cannot claim that $\mathcal{P%
}_{i}=(p,\mathcal{P}_{i}(\theta ))$, since it may happen that $\mathcal{P}%
_{i}(\theta )\equiv 0(\mod \mathcal{P}_{j})$ for some $\mathcal{P}%
_{j}\in \mathcal{T}$. So that in this situation $\mathcal{P}_{i}\mathcal{P}%
_{j}\mid (p,P(\theta ))$. Therefore we cannot apply directly the preceding
lemma. To remedy the situation, let $\theta _{j}\in A$ such that $\theta
_{j}\equiv 0(\mod \mathcal{P}_{j})$ for $j=s+1,...,r$ and let $\gamma
\in A$ such that $\gamma \equiv \theta (\mod \mathcal{P}_{i})$ for $%
i=1,...,s$ and $\gamma \equiv \theta _{j}(\mod \mathcal{P}_{j})$ for $%
j=s+1,...,r$. Let $\mathfrak{f}^{\prime }=(A:\mathbb{Z}\left[ \gamma \right]
)$. We apply the preceding lemma with $\gamma $ and $\mathfrak{f}^{\prime }$
replacing $\theta $ and $\mathfrak{f}$ respectively and for the same ideals $%
\mathcal{P}_{i}$ $(i=1,...,s)$. The conditions (i), (ii) and (iii) are
fulfilled and we deduce from the lemma that $\mathcal{P}_{i}\nmid \mathfrak{f%
}^{\prime }$; Thus $\mathcal{P}_{i}+\mathfrak{f}^{\prime }=A$ for $i=1,...,s$%
. Let $\alpha _{i}^{\prime }\in \mathcal{P}_{i}$ and $\beta _{i}^{\prime
}\in \mathfrak{f}^{\prime }$ such that $1=\alpha _{i}^{\prime }+\beta
_{i}^{\prime }$. Then $1=\overset{s}{\underset{i=1}{\prod }}(\alpha
_{i}^{\prime }+\beta _{i}^{\prime })=(\overset{s}{\underset{i=1}{\prod }}%
\alpha _{i}^{\prime })+\beta ^{\prime }$, where $\beta ^{\prime }\in 
\mathfrak{f}^{\prime }$, hence $1=\alpha ^{\prime }+\beta ^{\prime }$ where $%
\alpha ^{\prime }\in I$ and $\beta ^{\prime }\in \mathfrak{f}^{\prime }$.
Therefore $A=I+\mathfrak{f}^{\prime }$, thus $A=I+\mathbb{Z}\left[ \gamma %
\right] $. Let $\eta \in A$, then $\eta =\mu +v(\gamma )$, with $\mu \in I$
and $v(\gamma )\in \mathbb{Z}\left[ \gamma \right] $, hence $\eta \equiv
v(\gamma )(\mod I)\equiv v(\theta )(\mod I)$, thus $R=\mathbb{Z}%
\left[ \theta +I\right] $.\newline
In both cases, $\mathcal{T=\varnothing }$ and $\mathcal{T\neq \varnothing }$
we reach a stage in which $A=\mathbb{Z}\left[ \theta +I\right] $. Since $%
\theta +I$ is invertible in the finite ring $R$, then $\mathbb{Z}\left[
\theta +I\right] =\left \langle \left \langle \theta +I\right \rangle
\right
\rangle $.
\end{proof}
\subsection{\textbf{Faulty ideals }}
Recall the following notations stated in the introduction. Let $l$ be a
prime number and $f$ be a positive integer, then $N_{l}(f)$ denotes the
number of monic irreducible polynomials over $\mathbb{F}_{p}$ of degree $f$
and%
\begin{equation*}
\tau _{l}(f)=\left\{ 
\begin{array}{cc}
l & \text{if }f=1 \\ 
N_{l}(f)+1 & \text{if }f\geq 2%
\end{array}%
\right. 
\end{equation*}

\begin{definition}\label{definition5}
Let $K$ be a number field and $A$ its ring of integers. A nonzero ideal $I$
of $A$ is said to be faulty if there exist a prime number $p$ and a positive
integer $f$ such that $\tau _{p}(f)\leq I_{p}(f)$, where $I_{p}(f)$ denotes
the number of prime ideals of $A$ lying over $p\mathbb{Z}$, dividing $I$ and
having their residual degree equal to $f$.
\end{definition}

\begin{theorem}\label{theorem6}
Let $K$ be a number field, $A$ its ring of integers. Let $I$ be an ideal of $A$. Then, $A/I$ is coverable if and only if $%
I $ is faulty.
\end{theorem}
For the proof, we will use the following result.
\begin{lemma}\label{lemma7}
Let $K$ be a number field, $A$ its ring of integers, $p$ be a prime number, $%
\mathcal{P}_{1},...,\mathcal{P}_{s}$ be prime ideals of $A$ lying over $p%
\mathbb{Z}$ of residual degree $f_{1},...,f_{s}$ respectively. For any $%
i=1,...,s$, let $g_{i}(x)\in \mathbb{F}_{p}\left[ x\right] $ (not
necessarily distinct), be monic, irreductible over $\mathbb{F}_{p}$ of
degree $d_{i}$ dividing $f_{i}$. Let $h_{1},...,h_{s}$ be arbitrary positive
integers. Then there exists $\theta \in A$ such that $\upsilon _{\mathcal{P}%
_{i}}(g_{i}(\theta ))=h_{i}$.
\end{lemma}

\begin{proof}
See \cite{En}.
\end{proof}

\begin{proof}
(Proof of the Theorem \ref{theorem6}). ($\Longrightarrow $). We prove the contrapositive.
Suppose that $I$ is not faulty. By Proposition \ref{proposition4}, it suffices to prove that there
exists an element $\gamma \in A$ such that $\left \langle \left \langle 
\overline{\gamma }\right \rangle \right \rangle =R$. For every $i=1,...,s$,
let $g_{i}(x)\in \mathbb{F}_{p}\left[ x\right] $ to be a monic and
irreducible polynomial of degree $f_{i}$. Since $I$ is not faulty and we can
choose these polynomials such that :
\begin{equation*}
\left \{ 
\begin{array}{cc}
g_{i}(x)\not \equiv g_{j}(x)(\mod p\mathbb{Z)} & \text{if }i\neq j \\ 
g_{i}(x)\not \equiv x(\mod p\mathbb{Z)}\text{ for }i=1,...,s & \text{if 
}f_{i}=1%
\end{array}%
\right.
\end{equation*}%

Then, by Lemma \ref{lemma7}, there exists an element $\gamma \in A$ such that for every 
$i=1,...,s$, we have $\nu _{\mathcal{P}_{i}}(g_{i}(\gamma ))=1$. This
element $\gamma $ satisfies the condition (2) of Proposition \ref{proposition4}. Hence, $%
R=\left \langle \left \langle \overline{\gamma }\right \rangle
\right
\rangle $.\newline
($\Longleftarrow $). Since $I$ is faulty, then there exist $i\in \left \{
1,...,s\right \} $ such that $\tau _{p}(f_{i})\leq I_{p}(f_{i}):=\lambda $.%
\newline
\underline{Case $f_{i}\geq 2$}. Let $\mathfrak{Q}_{1},...,\mathfrak{Q}%
_{\lambda }$ be the ideals among the $\mathcal{P}_{i}$ having their residual
degree equal to $f_{i}$. Fix $\theta \in A$. We are going to show that $%
R\neq \left \langle \left \langle \overline{\theta }\right \rangle
\right
\rangle $.\newline
a) If for some $j=1,...,\lambda ,\deg Irr(\theta +\mathfrak{Q}_{j},\mathbb{F}%
_{p},x)<f_{i}$, then $A/\mathfrak{Q}_{j}\neq \left \langle \left \langle
\theta +\mathfrak{Q}_{j}\right \rangle \right \rangle $, hence by
Proposition \ref{proposition4}, $R\neq \left \langle \left \langle \theta \right \rangle
\right \rangle $.\newline
b) Suppose that $\deg Irr(\theta +\mathfrak{Q}_{j},\mathbb{F}_{p},x)=f_{i}$
for all $j=1,...,\lambda $. Let $g_{i}(x)=Irr(\theta +\mathfrak{Q}_{j},%
\mathbb{F}_{p},x)$ for $i=1,...,\lambda $. Since $\lambda \geq \tau
_{p}(f_{i})>N_{p}(f_{i})$, then there exists $i$ and $j$, $i>j$ such that $%
g_{i}(x)=g_{j}(x)$, thus, by Proposition \ref{proposition4}, $R\neq \left \langle
\left
\langle \overline{\theta }\right \rangle \right \rangle $.\newline
\underline{Case $f_{i}=1$}. Here $\lambda =I_{p}(f_{i})=I_{p}(1)\geq p$. We
me suppose that $\mathcal{P}_{1},...,\mathcal{P}_{p}$ have their residual
degree equal to $1$. Then there exist $i,j\in \left \{ 1,...,p\right \}
,i\neq j$, such that $Irr(\theta +\mathcal{P}_{i},\mathbb{F}%
_{p},x)=Irr(\theta +\mathcal{P}_{j},\mathbb{F}_{p},x)$ or there exists $i\in
\left \{ 1,...,p\right \} $ such that $Irr(\theta +\mathcal{P}_{i},\mathbb{F}%
_{p},x)=x$. In the first case, $\left \langle \left \langle \overline{\theta 
}\right \rangle \right \rangle \neq R$ by Proposition \ref{proposition4}. In the second case, 
$\overline{\theta }$ is not a unit of $R$, hence $\left \langle
\left
\langle \overline{\theta }\right \rangle \right \rangle \varsubsetneq 
\mathbb{Z}\left[ \overline{\theta }\right] \subseteq R$.
\end{proof}

\begin{corollary}\label{corollary8}
Let $K$ be a number field, $A$ its ring of integers and $p$ a prime number.
Then $A/pA$ is coverable if and only if $p\mid I(K)$ or there are at least $%
p $ prime ideals of $A$ of the first residual degree lying over $p\mathbb{Z}$%
.
\end{corollary}

\begin{proof}
Apply Theorem \ref{theorem6} for $I=pA=\mathcal{P}_{1}^{e_{1}}...\mathcal{P}_{r}^{e_{r}}$.
\end{proof}

\begin{corollary}\label{corollary9}
Let $K$ be a number field, $A$ be its ring of integers and $I$ be an ideal
of $A$. Then $A/I$ is coverable if and only if $I$ is faulty.
\end{corollary}
For the proof of this result, we use the following lemma.
\begin{lemma}\label{lemma10}
Let $R=\prod \limits_{i=1}^{i=t}R_{i}$, where $\left \vert R_{i}\right \vert
=p_{i}^{n_{i}}$ for some pairwise distinct primes $p_{1},...,p_{t}$ and some
positive integers $n_{1},...,n_{t}$. Then,\newline
$R$ is coverable if and only if at least one $R_{i}$ is coverable.
\end{lemma}

\begin{proof}
See Corollary 2.4 \cite{We}.
\end{proof}

\begin{proof}
(Proof of Corollary \ref{corollary9}). Write $I$ in the form $I=\prod
\limits_{j=1}^{t}I_{j} $ where for every $j$, $I_{j}$ is a (finite) product
of powers of prime ideals of $A$ lying over a prime number $p_{j}$ and the $%
p_{j}$'s are pairwise distinct. Then, $A/I=\prod \limits_{i=1}^{t}A/I_{j}$
and for every $j$, the cardinality of $A/I_{j}$ is a power of $p_{j}$. Then,
by using Lemma \ref{lemma10}, the ring $A/I$ is coverable if and only if one of the $%
A/I_{j}$ is coverable. We then apply Theorem \ref{theorem6}.
\end{proof}
For the next result, we need the following lemma.
\begin{lemma}\label{lemma11}
Let $R$ be a finitely coverable ring. Then, there exists a two sided ideal $%
I $ of $R$ such that $R/I$ is finite, coverable and $\sigma (R)=\sigma (R/I)$%
.
\end{lemma}

\begin{proof}
See Proposition 1.2 \cite{We}.
\end{proof}

\begin{theorem}\label{theorem12}
Let $K$ be a number field and $A$ be its ring of integers. Then $A$ is
finitely coverable if and only if there exists a prime number $p$ such that $%
p\mid I(K)$, or there are at least $p$ prime ideals of $A$ of the first
residual degree lying over $p\mathbb{Z}$.
\end{theorem}

\begin{proof}
Suppose that $A$ is finitely coverable, then by Lemma \ref{lemma11} there exists an
ideal $I$ of $A$ such that $A/I$ is coverable. The conclusion follows from
Corollary \ref{corollary9}. Conversely, if $p\mid I(K)$ or there are at least $p$ prime
ideals of $A$ of residual degree equal to $1$ lying over $p\mathbb{Z}$, then 
$A/pA$ is finitely coverable by Corollary \ref{corollary8}. Hence $A$ is finitely coverable.
\end{proof}

\section{\textbf{Computation of }$\protect\sigma (A)$}

\subsection{\textbf{Maximal subrings of }$A/\mathcal{P}^{k}$}

Let $K$ be a number field, $A$ be its ring of integers and $k$ be a positive
integer. Let $\mathcal{P}$ be a prime ideal of $A$ lying over $p\mathbb{Z}$
where $p$ is a prime number and let $f$ be the residual degree of $\mathcal{P%
}$ over $p\mathbb{Z}$. We describe the maximal subrings of $A/\mathcal{P}%
^{k} $.\newline
Suppose that $k=1$. Since $A/\mathcal{P}$ is a field isomorphe to $\mathbb{F}%
_{q}=\mathbb{F}_{p^{f}}$, then any maximal subring $M$ of $A/\mathcal{P}$
has the foolowing form.\newline
If $f=1$, $\mathbb{F}_{q}=\mathbb{F}_{p}$ and then $M=\left \{ 0\right \} $.%
\newline
If $f\geq 2$, then $M=\mathbb{F}_{p^{f/l}}$ for some prime number $l$
dividing $f$.\newline
Conversely, these subrings are maximal.

From now, we suppose that $k\geq 2$.

\begin{lemma}\label{lemma13}
$A/\mathcal{P}^{k}$ is a local ring and $\mathcal{P}/\mathcal{P}^{k}$ is its
maximal ideal.
\end{lemma}

\begin{proof}
Let $\varphi :A/\mathcal{P}^{k}\longrightarrow A/\mathcal{P}$ be the
canonical epimorphism. Obviously $\ker (\varphi )=\mathcal{P}/\mathcal{P}%
^{k} $. Since $A/\mathcal{P}$ is a field, then $\mathcal{P}/\mathcal{P}^{k}$
is a maximal ideal of $A/\mathcal{P}^{k}$.\newline
Since $(A/\mathcal{P}^{k})^{\ast }=(A/\mathcal{P}^{k})\setminus (\mathcal{P}/%
\mathcal{P}^{k})$, then $A/\mathcal{P}^{k}$ is a local ring.
\end{proof}

\begin{lemma}\label{lemma14}
Let $B$ be a subring of $A$ containing $\mathcal{P}^{k}$. Then $B/\mathcal{P}%
^{k}$ is maximal in $A/\mathcal{P}^{k}$ if and only if $B$ is maximal in $A$.
\end{lemma}

\begin{proof}
This is due to the fact that there is an increasing bijection between the
set of the subrings of $A$ containing $\mathcal{P}^{k}$ and the subrings of $%
A/\mathcal{P}^{k}$.
\end{proof}

If $S$ is any subring of a ring $R$, let $(R:S)=\left \{ x\in R,xR\subset
S\right \} $. $(R:S)$ is said to be the conductor of $S$ in $R$.

\begin{lemma}\label{lemma15}
Let $B$ be a subring of $A$ containing $\mathcal{P}^{k}$. Then $(A:B)/%
\mathcal{P}^{k}=(A/\mathcal{P}^{k}:B/\mathcal{P}^{k})$.
\end{lemma}

\begin{proof}
Let $I=(A:B)$ and $J/\mathcal{P}^{k}=(A/\mathcal{P}^{k}:B/\mathcal{P}^{k})$,
where $J$ is an ideal of $A$ containing $\mathcal{P}^{k}$. Let $x\in I$ and $%
a+\mathcal{P}^{k}\in A/\mathcal{P}^{k}$. Then $(a+\mathcal{P}^{k})(x+%
\mathcal{P}^{k})=ax+\mathcal{P}^{k}\in B/\mathcal{P}^{k}$ (since $ax\in B$),
thus $x+\mathcal{P}^{k}\in J/\mathcal{P}^{k}$, hence $x\in J$. This proves $%
I\subset J$. For the reverse inclusion, let $j\in B$ and $a\in A$, then $(j+%
\mathcal{P}^{k})(a+\mathcal{P}^{k})\in B/\mathcal{P}^{k}$, thus $ja-b\in 
\mathcal{P}^{k}$ for some $b\in B$, hence $ja\in B$ and $j\in I$.
\end{proof}

\begin{lemma}\label{lemma16}
Let $R$ be a commutative ring and $S$ be a maximal subring and let $%
\mathfrak{f=(}R:S)$. Then $\mathfrak{f}$ is a prime ideal of $S$ (not
necessarily of $R$).
\end{lemma}

\begin{proof}
We only prove the primality of $\mathfrak{f}$. Let $a,b\in S$ such that $%
ab\in \mathfrak{f}$. Suppose that $a\notin \mathfrak{f}$, then $S+aR=R$,
hence $bS+abR=bR$. Since $abR\in \mathfrak{f}\subset S$ and $bS\subset S$,
then $bR\subset S$, thus $b\in \mathfrak{f}$.
\end{proof}

\begin{lemma}\label{lemma17}
Let $B$ be a maximal subring of $A$ containing $\mathcal{P}^{k}$ and $%
\mathfrak{f=(}A:B)$.\newline
Then $\mathfrak{f=}\mathcal{P}$ or $\mathfrak{f=}\mathcal{P}^{2}$.
\end{lemma}

\begin{proof}
By Lemma \ref{lemma15}, $\mathfrak{f/}\mathcal{P}^{k}=(A/\mathcal{P}^{k}:B/\mathcal{P}%
^{k})$. By Lemma \ref{lemma16}, $\mathfrak{f}$ is prime in $B$. Since $\mathfrak{f}%
\supset \mathcal{P}^{k}$, then $\mathfrak{f}=\mathcal{P}^{h}$ with $1\leq
h\leq k-1$. Notice that $h\neq k$, otherwise $\mathfrak{f/}\mathcal{P}%
^{k}=\left \{ 0\right \} $ which is impossible because the conductor
contains the index of $B/\mathcal{P}^{k}$ in $A/\mathcal{P}^{k}$. Suppose
that $h\geq 3$, which implies $k\geq 4$ and $\mathcal{P}^{2}\not \subset B$.
Then $A/\mathcal{P}^{h}=(B+\mathcal{P}^{2})/\mathcal{P}^{k}\simeq B/(%
\mathcal{P}^{2}\cap \mathcal{P}^{h})=B/\mathcal{P}^{h}$. Since $\mathcal{P}%
^{h}$ is prime in $B$, this isomorphism shows that $\mathcal{P}^{h}$ is
prime in $A$, which is a contradiction, thus $h\leq 2$, that is $\mathfrak{f=%
}\mathcal{P}$ or $\mathfrak{f=}\mathcal{P}^{2}$.\newline
Recall that $h\leq k-1$, thus if $k=2$ then $h=1$ and then $\mathfrak{f=}%
\mathcal{P}$.
\end{proof}

\begin{lemma}\label{lemma18}
Let $R$ be a local ring of characteristic $p^{N}$, where $p$ is a prime
number and $N$ is a positive integer. Let $\mathfrak{M}$ be its maximal
ideal. Suppose that $\mathfrak{M}$ is nilpotent and $R/\mathfrak{M}\simeq 
\mathbb{F}_{q}$, where $q$ is a power of $p$. Then $R^{\ast }$ contains a
unique subgroup $T(R)$ isomorphic to $\mathbb{F}_{q}^{\ast }$. Furthermore,
let $\hat{T}(R)=T(R)\cup \left \{ 0\right \} $, then $1+\mathfrak{M}$ is a
multiplicative group, $R^{\ast }\simeq T(R)\times (1+\mathfrak{M)}$ and $R=%
\hat{T}(R)+\mathfrak{M}$. This last equality means that any element of $R$
may be expressed, in a unique way, as a sum of an element of $\hat{T}(R)$
and an element of $\mathfrak{M}$.
\end{lemma}

\begin{proof}
See \cite{Mu}.
\end{proof}

We now state the result about the maximal subrings of $R=A/\mathcal{P}^{k}$
with $k\geq 2$. We have seen (Lemma \ref{lemma13}) that this ring is local and $%
\mathfrak{M=}\mathcal{P}/\mathcal{P}^{k}$ is its unique maximal ideal.
Obviously $\mathfrak{M}$ is nilpotent, since $\mathfrak{M}^{k}=\left \{
0\right \} $. Let $e=\upsilon _{\mathcal{P}}(p)$, then the characteristic of 
$R$ is equal to $p^{N}$ with $N=\left \lceil k/e\right \rceil $,
thus Lemma \ref{lemma18}\ applies for the ring $R=A/\mathcal{P}^{k}$.

\begin{theorem}\label{theorem19}
There is two kinds of maximal subrings of $A/\mathcal{P}^{k}$.\newline
(i) Rings of the form $B/\mathcal{P}^{k}$ such that $(A/\mathcal{P}^{k}:B/%
\mathcal{P}^{k})=\mathcal{P}/\mathcal{P}^{k}$. The canonical epimorphism $A/%
\mathcal{P}^{k}\longrightarrow A/\mathcal{P}$ is a one to one correspondance
between these rings and the maximal subfields of $\mathbb{F}_{q}=\mathbb{F}%
_{p^{f}}$. Moreover $(B/\mathcal{P}^{k})/(\mathcal{P}/\mathcal{P}^{k})\simeq 
\mathbb{F}_{p^{f/l}}$ for some unique prime divisor $l$ of $f$. We call
these rings ordinary maximal subrings.\newline
(ii) One ring of the form $B/\mathcal{P}^{k}$ such that $(A/\mathcal{P}%
^{k}:B/\mathcal{P}^{k})=\mathcal{P}^{2}/\mathcal{P}^{k}$ and $(B/\mathcal{P}%
^{k})/(\mathcal{P}^{2}/\mathcal{P}^{k})\simeq \mathbb{F}_{q}$. We call it
exceptional maximal subring.
\end{theorem}

\begin{proof}
Let $B/\mathcal{P}^{k}$ be a maximal subring of $R=A/\mathcal{P}^{k}$. By
Lemma \ref{lemma14}, $B$ is maximal in $A$. By Lemma \ref{lemma15}, $(A/\mathcal{P}^{k}:B/\mathcal{%
P}^{k})=(A:B)/\mathcal{P}^{k}$. Lemma \ref{lemma17} implies $(A:B)=\mathcal{P}$ or $%
(A:B)=\mathcal{P}^{2}$, hence $(A/\mathcal{P}^{k}:B/\mathcal{P}^{k})=%
\mathcal{P}/\mathcal{P}^{k}$ or $(A/\mathcal{P}^{k}:B/\mathcal{P}^{k})=%
\mathcal{P}^{2}/\mathcal{P}^{k}$. The second case occurs only when $k\geq 3$.%
\newline
Suppose that we have the first case $(A/\mathcal{P}^{k}:B/\mathcal{P}^{k})=%
\mathcal{P}/\mathcal{P}^{k}$. The canonical surjection $\varphi :A/\mathcal{P%
}^{k}\longrightarrow A/\mathcal{P}$ induces a bijection between the set of
maximal subrings $B/\mathcal{P}^{k}$ of $A/\mathcal{P}$ containing $\mathcal{%
P}/\mathcal{P}^{k}$ onto the set of maximal subrings of $A/\mathcal{P}%
^{k}\simeq \mathbb{F}_{q}=\mathbb{F}_{p^{f}}$.\ Such rings are finite,
without zero divisors, hence they are fields, thus the bijectiion is onto
the set $\left \{ \mathbb{F}_{p^{f/l}},l\text{ prime, }l\mid f\right \} $.
Moreover ($B/\mathcal{P}^{k})/(\mathcal{P}/\mathcal{P}^{k})\simeq \mathbb{F}%
_{p^{f/l}}$ for some unique $l$.\newline
Suppose that we have the second case $(A/\mathcal{P}^{k}:B/\mathcal{P}^{k})=%
\mathcal{P}^{2}/\mathcal{P}^{k}$. We have seen that $R$ is a local ring with
maximal ideal $\mathfrak{M}=\mathcal{P}/\mathcal{P}^{k}$ and characteristic $%
p^{N}$, where $N=\left \lceil k/e\right \rceil $. Also $\mathfrak{M}$ is
nilpotent ($\mathfrak{M}^{k}=\left \{ 0\right \} $) and $R/\mathfrak{M}%
\simeq \mathbb{F}_{q}$. Lemma \ref{lemma18} implies that $R^{\ast }$ contains a unique
subgroup $T(R)$ isomorphic to $\mathbb{F}_{q}^{\ast }$, $R^{\ast }\simeq
T(R)\times (1+\mathfrak{M)}$ and $R=\hat{T}(R)+\mathfrak{M}$ where $\hat{T}%
(R)=T(R)\cup \left \{ 0\right \} $. Let $R_{1}=B/\mathcal{P}^{k}\subset R$. $%
R_{1}$ is also a local ring with maximal ideal $\mathcal{P}^{2}/\mathcal{P}%
^{k}=\mathfrak{M}^{2}$ and we have $(B/\mathcal{P}^{k})/(\mathcal{P}^{2}/%
\mathcal{P}^{k})\simeq B/\mathcal{P}^{2}$. Since $B$ is maximal in $A$ and $%
\mathcal{P}\not \subset B$, then $B+\mathcal{P=A}$, hence $A/\mathcal{P=(}B+%
\mathcal{P)}/\mathcal{P}\simeq B/(B\cap \mathcal{P)}=B/\mathcal{P}^{2}$,
thus $B/\mathcal{P}^{2}\simeq \mathbb{F}_{q}$. Lemma 6 applies for $R_{1}$, $%
R_{1}^{\ast }$ contains a unique subgroup $T(R_{1})$ isomorphic to $\mathbb{F%
}_{q}$, thus $T(R_{1})=T(R),R_{1}^{\ast }\simeq T(R_{1})\times (1+\mathfrak{M%
}^{2})$ and $R_{1}=\hat{T}(R_{1})+\mathfrak{M}^{2}=\hat{T}(R)+\mathfrak{M}%
^{2}$. This shows that $R_{1}$ is unique.
\end{proof}

We may describe the maximal subrings of $A/\mathcal{P}^{k}$ in the following
way.

\begin{lemma}\label{lemma20}
Let $\mathcal{P}$ be a prime ideal of $A$ and $k$ be a positive integer. Let 
$S$ be a complete system of representatives of $A$ modulo $\mathcal{P}$ such
that $0\in S$. Let $\pi \in \mathcal{P\setminus P}^{2}$, then\newline
\begin{equation*}
\hat{S}=\left \{ s_{0}+s_{1}\pi +...+s_{k-1}\pi ^{k-1},s_{i}\in
S,i=0,...,k-1\right \}
\end{equation*}%
\newline
is a complete system of representatives of $A$ modulo $\mathcal{P}^{k}$.
\end{lemma}

\begin{proof}
See 8.1.B \cite{Ri}.
\end{proof}

\begin{remark}\label{remark21}
Let $B/\mathcal{P}^{k}$ be an ordinary maximal subring of $A/\mathcal{P}^{k}$
such that\newline
$(B/\mathcal{P}^{k})/(\mathcal{P}/\mathcal{P}^{k})\simeq \mathbb{F}%
_{p^{f/l}} $ for some prime divisor $l$ of $f$, then,
\begin{equation*}
\left \{ s_{0}+s_{1}\pi +...+s_{k-1}\pi ^{k-1},s_{i}\in S,s_{0}+\mathcal{%
P\in }\mathbb{F}_{p^{f/l}}\right \}
\end{equation*}%
is a complete system of representatives of $B$ modulo $\mathcal{P}^{k}$.%
\newline
Let $B/\mathcal{P}^{k}$ be the unique exceptional maximal subring of $A/%
\mathcal{P}^{k}$ such that $(B/\mathcal{P}^{k})/(\mathcal{P}^{2}/\mathcal{P}%
^{k})\simeq \mathbb{F}_{q}$, then $\left \{ s_{0}+s_{2}\pi
^{2}+...+s_{k-1}\pi ^{k-1},s_{i}\in S,i=0,2,...,k-1\right \} $ is complete
system of representatives of $B$ modulo $\mathcal{P}^{k}$.
\end{remark}

\subsection{\textbf{Maximal subrings of }$A/\mathcal{P}_{1}^{k_{1}}\times
...\times A/\mathcal{P}_{s}^{k_{s}}$}

Let $p$ be a prime number and $\mathcal{P}_{1},...,\mathcal{P}_{s}$ be prime
ideals of $A$ lying over $p\mathbb{Z}$. We will describe the maximal
subrings of $R$. For this, we will use the following.

\begin{lemma}\label{lemma22}
Let $R_{1},...,R_{m}$ be rings, then $R=\overset{m}{\underset{i=1}{\prod }}%
R_{i}$ has a maximal subring if and only if one of the following conditions
holds.\newline
(1) For some $i$, $R_{i}$ has a maximal subring.\newline
(2) There exist $i\neq j$ and maximal ideals $Q_{i}$ of $R_{i}$ and $Q_{j}$
of $R_{j}$ such that $R_{i}/Q_{i}\simeq R_{j}/Q_{j}$. (Thus in this case,
there exist distinct maximal ideals $Q$ and $T$ in $R$ with $R/Q\simeq R/T$.)
\end{lemma}

\begin{proof}
See \ Corollary 2.3 \cite{Az.K.}.
\end{proof}

Notice that $Q=R_{1}\times ...R_{i-1}Q_{i}\times R_{i+1},...,R_{m}$ and $%
T=R_{1}\times ...R_{j-1}Q_{j}\times R_{j+1},...,R_{m}$ fulfill the condition 
$R/Q\simeq R/T$.

\begin{theorem}\label{theorem23}
Let $p$ be a prime number, $\mathcal{P}_{1},...,\mathcal{P}_{s}$ be prime
ideals of $A$ lying over $p\mathbb{Z}$ and $k_{1},...k_{s}$ be positive
integers. let $R=\underset{i=1}{\overset{s}{\prod }}A/\mathcal{P}%
_{i}^{k_{i}} $ and let $M$ be a maximal subring of $R$. Then $M$ has one of
the following forms.\newline
(i) $M=A/\mathcal{P}_{1}^{k_{1}}\times ...\times A/\mathcal{P}%
_{i-1}^{k_{i-1}}\times M_{i}\times A/\mathcal{P}_{i+1}^{k_{i+1}}\times
...\times A/\mathcal{P}_{s}^{k_{s}}$ for some $i\in \left \{
1,...,s\right
\} $ and where $M_{i}$ is an ordinary maximal subring of $A/%
\mathcal{P}_{i}^{k_{i}}$.\newline
We call them ordinary type 1.\newline
(ii) $M=A/\mathcal{P}_{1}^{k_{1}}\times ...\times A/\mathcal{P}%
_{i-1}^{k_{i-1}}\times M_{i}\times A/\mathcal{P}_{i+1}^{k_{i+1}}\times
...\times A/\mathcal{P}_{s}^{k_{s}}$ for some $i\in \left \{
1,...,s\right
\} $ and where $M_{i}$ is the unique exceptional maximal
subring of $A/\mathcal{P}_{i}^{k_{i}}$.\newline
We call it exceptional type 1.\newline
(iii) $M=M_{ij}^{\overline{\sigma }}$ for some $(i,j),1\leq i<j\leq s$ where%

\begin{equation*}
M_{ij}^{\overline{\sigma }}=\left \{ (a_{1}+\mathcal{P}%
_{1}^{k_{1}},...,a_{s}+\mathcal{P}_{s}^{k_{s}})\in
R,a_{j}+P_{j}^{k_{j}}+P_{j}/P_{j}^{k_{j}}=\overline{\sigma }%
(a_{i}+P_{i}^{k_{i}}+P_{i}/P_{i}^{k_{i}})\right \}
\end{equation*}%

for some isomorphism $\overline{\sigma }:R_{i}=(A/\mathcal{P}%
_{i}^{k_{i}})/(P_{i}/P_{i}^{k_{i}})\longrightarrow R_{j}=(A/\mathcal{P}%
_{j}^{k_{j}})/(P_{j}/P_{j}^{k_{j}})$ defined by the following diagramm%

\begin{equation*}
\begin{array}{ccc}
R_{i} & \overset{\overline{\sigma }}{\longrightarrow } & R_{j} \\ 
\downarrow \beta &  & \downarrow \alpha \\ 
A/\mathcal{P}_{i}\simeq \mathbb{F}_{q} & \overset{\sigma }{\longrightarrow }
& A/\mathcal{P}_{j}\simeq \mathbb{F}_{q}%
\end{array}%
\end{equation*}%
where $\alpha $ et $\beta $ are the canonical isomorphisms from $R_{i}$ and $%
R_{j}$ repectively onto $\mathbb{F}_{q}$.\newline
We call these rings maximal subrings of type 2.
\end{theorem}

\begin{proof}
The maximal subrings described in (i) or (ii) correspond to the case (1) of
Lemma \ref{lemma22}, while the maximals subrings of type 2 correspond to the case (2)
of the same Lemma.
\end{proof}

Let $p$ be a prime number and $R=\underset{i=1}{\overset{t}{\prod }}R_{i}$,
where for any $i,R_{i}=\underset{j=1}{\overset{t_{i}}{\prod }}A/\mathcal{P}%
_{ij}^{k_{ij}}$ and all the $\mathcal{P}_{ij}$ $(j=1,...,t_{i})$ are prime
ideals of $A$ lying over $p\mathbb{Z}$ and having the same residual degree $%
f_{i}$. here we deal with $\sigma (R)$. We recall a result from \cite{We}.

\begin{lemma}\label{lemma24}
Let $R=\underset{i=1}{\overset{s}{\prod }}R_{i}$ where each $R_{i}$ is a
finite ring. Assume that all maximal subrings of $R$ have the form\newline
\begin{equation*}
M=R_{1}\times ...\times R_{i-1}\times M_{i}\times R_{i+1}\times ...\times
R_{s}
\end{equation*}%
\newline
where $M_{i}$ is a maximal subring of $R_{i}$.\newline
Then $\sigma (R)=\overset{s}{\underset{i=1}{\inf }}\sigma (R_{i})$.
\end{lemma}

\begin{proof}
See Theorem 2.2 in \cite{We}
\end{proof}

\begin{theorem}\label{theorem25}
Let $R$ be the ring described in the begining of this section.\newline
Then $\sigma (R)=\overset{t}{\underset{i=1}{\inf }}\sigma (R_{i})$.
\end{theorem}

\begin{proof}
We rewrite $R$ in the form $R=\underset{h=1}{\overset{s}{\prod }}A/\mathcal{P%
}_{h}^{k_{h}}$. By Theorem \ref{theorem23}, any maximal subring $M$ of $R$ has the form $%
M=A/\mathcal{P}_{1}^{k_{1}}\times ...\times N_{u}\times ...\times A/\mathcal{%
P}_{s}^{k_{s}}$ where $N_{u}$ is a maximal subring of $A/\mathcal{P}%
_{u}^{k_{u}}$, or $M=N_{uv}^{\overline{\sigma }}$ where $\overline{\sigma }$
relates two indices $u$ and $v$ for which $\mathcal{P}_{u}$ and $\mathcal{P}%
_{v}$ has the same residual degree, say $f$, so that $A/\mathcal{P}%
_{u}\simeq A/\mathcal{P}_{v}\simeq \mathbb{F}_{p^{f}}$.\newline
In the first case, let $i\in \left \{ 1,..,t\right \} $ and $j\in \left \{
1,..,t_{i}\right \} $ such that $\mathcal{P}_{u}=\mathcal{P}_{ij}$. Then%

\begin{equation*}
M=R_{1}\times ...\times R_{i-1}\times (A/\mathcal{P}_{i1}^{k_{i1}}\times
...\times N_{u}\times ...\times A/\mathcal{P}_{it_{i}}^{k_{it_{i}}})\times
R_{i+1}\times ...R_{t}
\end{equation*}%

Obviously, $M_{i}=A/\mathcal{P}_{i1}^{k_{i1}}\times ...\times N_{u}\times
...\times A/\mathcal{P}_{it_{i}}^{k_{it_{i}}}$ is a maximal subring of $%
R_{i} $, then $M=R_{1}\times ...\times M_{i}\times ...\times R_{t}$ and the
conclusion of the theorem follows from Lemma \ref{lemma24}.\newline
In the second case, there exists $i\in \left \{ 1,..,t\right \} $ and $%
j_{1},j_{2}\in \left \{ 1,..,t_{i}\right \} ,j_{1}\neq j_{2}$, such that $%
\mathcal{P}_{u}=\mathcal{P}_{ij_{1}}$ and $\mathcal{P}_{v}=\mathcal{P}%
_{ij_{2}}$. Let $M_{i}$ be the subring of $R_{i}$ fomed by the elements $%
(a_{1}+\mathcal{P}_{i1}^{k_{i1}},...,a_{t_{i}}+\mathcal{P}%
_{t_{i}}^{k_{t_{i}}})$ such that
\begin{equation*}
a_{ij_{2}}+P_{ij_{2}}^{k_{ij_{2}}}+P_{ij_{2}}/P_{ij_{2}}^{k_{ij_{2}}}=%
\overline{\sigma }%
(a_{ij_{1}}+P_{ij_{1}}^{k_{ij_{1}}}+P_{ij_{1}}/P_{ij_{1}}^{k_{ij_{1}}})\text{%
.}
\end{equation*}%

Then
\begin{equation*}
M=R_{1}\times ...\times M_{i}\times ...\times R_{t}
\end{equation*}%
and the conclusion of the theorem follows.
\end{proof}
\subsection{\textbf{The minimal covering of} $A$ {and the value of} $\sigma(A)$}

\begin{theorem}\label{theorem26}
Let $p$ be a prime number, $K$ be a un number field and $A$ be its ring of
integers. Let $s$ be a positive integer and $\mathcal{P}_{1},...,\mathcal{P}%
_{s}$ be prime ideals of $A$ lying over $p\mathbb{Z}$ and having the same
residual degree $f$. Suppose that $s\geq \tau =\tau (f)$. Let \newline
\begin{eqnarray*}
R &=&A/\mathcal{P}_{1}^{k_{1}}\times ...\times A/\mathcal{P}_{1}^{k_{s}} \\
R_{1} &=&A/\mathcal{P}_{1}\times ...\times A/\mathcal{P}_{s} \\
R_{0} &=&A/\mathcal{P}_{1}\times ...\times A/\mathcal{P}_{\tau }
\end{eqnarray*}%
\newline
Then $\sigma (R)=\sigma (R_{1})=\sigma (R_{0})=f\left( 
\begin{array}{c}
\tau \\ 
2%
\end{array}%
\right) +\tau \nu (f)$.
\end{theorem}

For the proof of this theorem, we will use the following two lemmas.

\begin{lemma}\label{lemma27}
Under the hypotheses of Theorem \ref{theorem26}, we have $\sigma (R_{1})=\sigma (R_{0})$.
\end{lemma}

\begin{proof}
See Theorem 5.4 \cite{We}.
\end{proof}

\begin{lemma}\label{lemma28}
Keep the hypothesis of Theorem \ref{theorem26} and let $R=\overset{m}{\underset{i=1}{\cup 
}}S_{i}$ be a minimal covering of $R$, then no $S_{i}$ is an exceptional
maximal subring of type $1$ of $R$.
\end{lemma}

\begin{proof}
We prove that given a covering of $R$, $R=\overset{m}{\underset{i=1}{\cup }}%
S_{i}$, where $S_{i}$ is an exceptional maximal subring of type $1$ of $R$
for $i=1,...,m_{0}$ with $m_{0}<m$, then $S_{1}\subset \overset{m}{\underset{%
i=2}{\cup }}S_{i}$. This will imply that in a minimal covering of $R$, no
component is an exceptional maximal subring of type $1$ of $R$.\newline
$S_{1}$ has the form
\begin{equation*}
S_{1}=A/\mathcal{P}_{1}^{k_{1}}\times ...\times A/\mathcal{P}%
_{j-1}^{k_{j-1}}\times B_{j}/\mathcal{P}_{j}^{k_{j}}\times A/\mathcal{P}%
_{j+1}^{k_{j+1}}\times ...\times A/\mathcal{P}_{s}^{k_{s}}
\end{equation*}%

for some $j\in \left \{ 1,...,s\right \} $ and where $B_{j}$ is a maximal
subring of $A$ such that\newline
\begin{equation*}
(B_{j}/\mathcal{P}_{j}^{k_{j}})/(\mathcal{P}_{j}^{2}/\mathcal{P}%
_{j}^{k_{j}})\simeq B_{j}/\mathcal{P}_{j}^{2}\simeq \mathbb{F}_{q}\text{
(see Theorem \ref{theorem23}).}
\end{equation*}%
\newline
Let $\overset{\rightarrow }{a}=(a_{1}+\mathcal{P}_{1}^{k_{1}},...,b_{j}+%
\mathcal{P}_{j}^{k_{j}},...,a_{s}+\mathcal{P}_{s}^{k_{s}})$ be an element of 
$S_{1}$ where $b_{j}\in B_{j}$ and $a_{h}\in A$ for $h\neq j$. Let $\pi \in 
\mathcal{P}_{j}\setminus \mathcal{P}_{j}^{2}$. Since $A=B_{j}+\mathcal{P}%
_{j} $ and $\mathcal{P}_{j}\nsubseteq B_{j}$, then there exists $a_{j}\in
A\setminus B_{j}$ suche that $a_{j}=b_{j}+\pi $. Let $\overset{\rightarrow }{%
a^{\prime }}=(a_{1}+\mathcal{P}_{1}^{k_{1}},...,a_{j}+\mathcal{P}%
_{j}^{k_{j}},...,a_{s}+\mathcal{P}_{s}^{k_{s}})$. This element must belong
to some maximal subring, say $S_{h}$ with $h\in \left \{ 1,...,s\right \} $.
The assumptions on $a_{j}$ shows that $\overset{\rightarrow }{a^{\prime }}%
\notin S_{1}$, that is $h\neq 1$.\newline
\underline{Case $S_{h}$ is of type 1} Suppose that
\begin{equation*}
S_{h}=A/\mathcal{P}_{1}^{k_{1}}\times ...\times A/\mathcal{P}%
_{u-1}^{k_{u-1}}\times C_{u}/\mathcal{P}_{u}^{k_{u}}\times ...\times C_{s}/%
\mathcal{P}_{s}^{k_{s}}
\end{equation*}%

with $u\neq j$, then $\overset{\rightarrow }{a}\in S_{h}$ and we are done. If%

\begin{equation*}
S_{h}=A/\mathcal{P}_{1}^{k_{1}}\times ...\times A/\mathcal{P}%
_{j-1}^{k_{h-1}}\times C_{j}/\mathcal{P}_{j}^{k_{j}}\times A/\mathcal{P}%
_{j+1}^{k_{j+1}}\times ...\times C_{s}/\mathcal{P}_{s}^{k_{s}}
\end{equation*}%
where $j$ is the index determined above and where $C_{j}/\mathcal{P}%
_{j}^{k_{j}}$ is a ordinary maximal subring of $A/\mathcal{P}_{j}^{k_{j}}$,
then
\begin{equation*}
b_{j}+\mathcal{P}_{j}^{k_{j}}+\mathcal{P}_{j}/\mathcal{P}_{j}^{k_{j}}=a_{j}-%
\pi +\mathcal{P}_{j}^{k_{j}}+\mathcal{P}_{j}/\mathcal{P}_{j}^{k_{j}}=a_{j}+%
\mathcal{P}_{j}^{k_{j}}+\mathcal{P}_{j}/\mathcal{P}_{j}^{k_{j}}
\end{equation*}%

thus $b_{j}+\mathcal{P}_{j}^{k_{j}}\in C_{j}/\mathcal{P}_{j}^{k_{j}}$ and
then $\overset{\rightarrow }{a}\in S_{h}$.\newline
\underline{Case $S_{h}$ is of type 2} Suppose that $S_{h}=\Delta
_{uv}^{\sigma }$ for some $\sigma $.\newline
If $u,v\neq j$, then obviously $\overset{\rightarrow }{a}\in S_{h}$.\newline
It remains to consider the cases $S_{h}=\Delta _{uj}^{\sigma }$ with $u<j$
and $S_{h}=\Delta _{uv}^{\sigma }$ with $j<v$, for some $\sigma $. Since the
proofs are similar for the two possibilities, we only look at the case $%
S_{h}=\Delta _{uj}^{\sigma }$.\newline
We have $\overline{\sigma }(a_{u}+\mathcal{P}_{u}^{k_{u}}+\mathcal{P}_{u}/%
\mathcal{P}_{u}^{k_{u}})=a_{j}+\mathcal{P}_{j}^{k_{j}}+\mathcal{P}_{j}/%
\mathcal{P}_{j}^{k_{j}}$, where $\overline{\sigma }$ is defined in Theorem
\ref{theorem23}. Since $a_{j}=b_{j}+\pi $, then
\begin{equation*}
a_{j}+\mathcal{P}_{j}^{k_{j}}+\mathcal{P}_{j}/\mathcal{P}_{j}^{k_{j}}=b_{j}+%
\pi +\mathcal{P}_{j}^{k_{j}}+\mathcal{P}_{j}/\mathcal{P}_{j}^{k_{j}}=b_{j}+%
\mathcal{P}_{j}^{k_{j}}+\mathcal{P}_{j}/\mathcal{P}_{j}^{k_{j}}
\end{equation*}%

hence
\begin{equation*}
\overline{\sigma }(a_{u}+\mathcal{P}_{u}^{k_{u}}+\mathcal{P}_{u}/\mathcal{P}%
_{u}^{k_{u}})=b_{j}+\mathcal{P}_{j}^{k_{j}}+\mathcal{P}_{j}/\mathcal{P}%
_{j}^{k_{j}}
\end{equation*}%

This shows that $\overset{\rightarrow }{a}\in S_{h}$.
\end{proof}

\begin{proof}
(Proof of the theorem). We first prove $\sigma (R_{0})=f\left( 
\begin{array}{c}
\tau \\ 
2%
\end{array}%
\right) +\tau \nu (f)$. Obviously, $R_{0}$ is the union of all its maximal
subrings. The number of these is equal to $f\left( 
\begin{array}{c}
\tau \\ 
2%
\end{array}%
\right) +\tau \nu (f)$. Thus $\sigma (R_{0})\leq f\left( 
\begin{array}{c}
\tau \\ 
2%
\end{array}%
\right) +\tau \nu (f)$.\newline
Suppose that in a given covering of $R_{0}$ some ordinary maximal subring $M$
of the first type is omitted. $M$ has the form
\begin{equation*}
M=A/\mathcal{P}_{1}\times ...\times A/\mathcal{P}_{i-1}\times M_{i}\times A/%
\mathcal{P}_{i+1}\times ...\times A/\mathcal{P}_{\tau }
\end{equation*}%

where
\begin{equation*}
M_{i}=\left \{ 
\begin{array}{ccc}
\mathbb{F}_{p^{f/l}}\text{ for some prime divisor }l\text{ of }f & \text{if}
& f\geq 2 \\ 
\left \{ 0\right \} & \text{if} & f=1%
\end{array}%
\right.
\end{equation*}%

If $f\geq 2$, let $g_{1}(x),...,g_{i-1}(x),g_{i+1}(x),...,g_{\tau }(x)$ be
the complete list of the monic irreducible polynomials in $\mathbb{F}_{p}%
\left[ x\right] $, of degree $f$. For each of them, select a root $\theta
_{j}\in A/\mathcal{P}_{j}$. Let $\theta _{i}$ be a generator of $\mathbb{F}%
_{p^{f/l}}\subset A/\mathcal{P}_{i}$. It is seen that the element $(\theta
_{1},...,\theta _{\tau })$ belongs to $M$ but not to any other maximal
subring of $R_{0}$ of the first type nor of the second type; hence a
contradiction.\newline
If $f=1$ (thus $\tau =p$), then $(1,...,i-1,0,i+1,...,p-1,p)$ belongs to $M$
but not to any other maximal subring, hence a contradiction.\newline
Suppose that in given covering some maximal subring of the second type of $%
R_{0}$ is omitted. Let $M_{ij}^{\overline{\sigma }}=\left \{ (a_{1}+\mathcal{%
P}_{1},...,a_{\tau }+\mathcal{P}_{\tau })\in R_{0},a_{j}+\mathcal{P}_{j}=%
\overline{\sigma }(a_{i}+\mathcal{P}_{i})\right \} ,i<j$, be this ring. Let $%
g_{1}(x),...,g_{i-1}(x),g_{i+1}(x),...,g_{\tau }(x)$ be the list of the
monic irreducible polynomials of degree $f$ over $\mathbb{F}_{p}$. For each
of these polynomials select a root $\theta _{h}\in A/\mathcal{P}_{h}$. It is
seen that $(\theta _{1},...,\theta _{i},...,\theta _{j-1},\overline{\sigma }%
(\theta _{i}),\theta _{j+1},...\theta _{\tau })\in M_{ij}^{\overline{\sigma }%
}$ but this element does not belong to any other maximal subring of $R_{0}$,
hence a contradiction. It follows that $\sigma (R_{0})=f\left( 
\begin{array}{c}
\tau \\ 
2%
\end{array}%
\right) +\tau \nu (f)$.\newline
This result was proved by Werner but our proof is slighly different.\newline
By Lemma \ref{lemma27}, we have $\sigma (R_{1})=\sigma (R_{0})$. We prove that $\sigma
(R)=\sigma (R_{1})$. Consider the canonical epimorphism\newline
\begin{equation*}
\Phi :A/\mathcal{P}_{1}^{k_{1}}\times ...\times A/\mathcal{P}%
_{s}^{k_{s}}\longrightarrow A/\mathcal{P}_{1}\times ...\times A/\mathcal{P}%
_{s}
\end{equation*}%
\newline
Its kernel is equal to $I=\mathcal{P}_{1}/\mathcal{P}_{1}^{k_{1}}\times
...\times \mathcal{P}_{s}/\mathcal{P}_{s}^{k_{s}}$. Let $R=\overset{m}{%
\underset{i=1}{\cup }}S_{i}$ be a minimal covering of $R$. If it is proved
that $I\subset S_{i}$ for $i=1,...,m$, then we may conclude that $\sigma
(R)=\sigma (R_{1})$. By Lemma \ref{lemma11}, no $S_{i}$ is an exceptional maximal
subring of type 1 of $R$. So we have to show that $I\subset S_{i}$, where $%
S_{i}$ is an ordinary maximal subring of $R$ of type 1 or a maximal subring
of $R$ of type 2.\newline
Suppose that $S_{i}=A/\mathcal{P}_{1}^{k_{1}}\times ...\times B_{h}/\mathcal{%
P}_{h}^{k_{h}}\times ...\times A/\mathcal{P}_{s}^{k_{s}}$. Here $B_{h}$ is a
maximal subring of $A$ and its conductor is equal to $\mathcal{P}_{h}$, thus 
$\mathcal{P}_{h}/\mathcal{P}_{h}^{k_{h}}$ is an ideal of $B_{h}/\mathcal{P}%
_{h}^{k_{h}}$ and $A/\mathcal{P}_{1}^{k_{1}}\times ...\times \mathcal{P}_{h}/%
\mathcal{P}_{h}^{k_{h}}\times ...\times A/\mathcal{P}_{s}^{k_{s}}$ is an
ideal of $S_{i}$. It follows that $I\subset S_{i}$.\newline
Suppose that $S_{i}=\Delta _{uv}^{\sigma }$, with $1\leq u<v\leq s$. Here 

\begin{equation*}
\Delta _{uv}^{\sigma }=(\underset{j\neq u,v}{\prod }A/\mathcal{P}%
_{j}^{k_{j}})\times \Gamma _{uv}^{\sigma }
\end{equation*}%

where $\Gamma _{uv}^{\sigma }$ is the subring of $A/\mathcal{P}%
_{u}^{k_{u}}\times A/\mathcal{P}_{v}^{k_{v}}$ formed from the elements $%
(a_{u}+\mathcal{P}_{u}^{k_{u}},a_{v}+\mathcal{P}_{v}^{k_{v}})$ such that 
\begin{equation*}
a_{v}+\mathcal{P}_{v}^{k_{v}}+\mathcal{P}_{v}/\mathcal{P}_{v}^{k_{u}}=%
\overline{\sigma }(a_{u}+\mathcal{P}_{u}^{k_{u}}+\mathcal{P}_{v}/\mathcal{P}%
_{u}^{k_{u}})
\end{equation*}%

Let $(a_{u}+\mathcal{P}_{u}^{k_{u}},a_{v}+\mathcal{P}_{v}^{k_{v}})\in 
\mathcal{P}_{u}/\mathcal{P}_{u}^{k_{u}}\times \mathcal{P}_{v}/\mathcal{P}%
_{v}^{k_{v}}$, then $a_{u}+\mathcal{P}_{u}^{k_{u}}+\mathcal{P}_{v}/\mathcal{P%
}_{u}^{k_{u}}=0$ and $a_{v}+\mathcal{P}_{v}^{k_{v}}+\mathcal{P}_{v}/\mathcal{%
P}_{v}^{k_{u}}=0$. Therefore
\begin{equation*}
\overline{\sigma }(a_{u}+\mathcal{P}_{u}^{k_{u}}+\mathcal{P}_{v}/\mathcal{P}%
_{u}^{k_{u}}=a_{v}+\mathcal{P}_{v}^{k_{v}}+\mathcal{P}_{v}/\mathcal{P}%
_{v}^{k_{u}}
\end{equation*}%

This implies that $(\underset{j\neq u,v}{\prod }A/\mathcal{P}%
_{j}^{k_{j}})\times \mathcal{P}_{u}/\mathcal{P}_{u}^{k_{u}}\times \mathcal{P}%
_{v}/\mathcal{P}_{v}^{k_{v}}\subset S_{i}$. We deduce that $I=\mathcal{P}%
_{1}/\mathcal{P}_{1}^{k_{1}}\times ...\times \mathcal{P}_{s}/\mathcal{P}%
_{s}^{k_{s}}\subset S_{i}$.
\end{proof}

\begin{theorem}\label{theorem29}
Let $K$ be a number field and $A$ be its ring of integers. Then\newline
1) $\sigma (A)=\underset{I}{\inf }\sigma (A/I)$ where $I$ describes the set
of faulty ideals of $A$.\newline
2) Let $\mathfrak{P}$ be the set of prime numbers $p$ such that \newline
(i) $p$ is a common index divisor in $K$;\newline
or\newline
(ii) there are at least $p$ prime ideals of $A$ lying over $p\mathbb{Z}$
having their residual degree equal to $1$.\newline
For any $p\in \mathfrak{P}$, let $\mathfrak{F}(p)$ be the set of residual
degrees $f$ of prime ideals appearing in the splitting of $p$ such that $%
N_{p(f)}<\lambda _{p}(f)$ for $f\geq 2$ and $p=N_{p}(f)\leq \lambda _{p}(f)$
if $f=1$.\newline
Then
\begin{equation*}
\sigma (A)=\underset{p\in \mathfrak{P}}{\inf }(\underset{\mathfrak{f\in F}(p)%
}{\inf }\sigma (A/\mathcal{P}_{1}\times ...\times A/\mathcal{P}_{\tau (f)}))
\end{equation*}%
where $\mathcal{P}_{1},...,\mathcal{P}_{\tau (f)}$ is any set of prime
ideals of $A$ lying over $p\mathbb{Z}$ with residual degree equal to $f$.%
\newline
3) $\sigma (A)=\underset{p\in \mathfrak{P}}{\inf }\sigma (A/pA)$.
\end{theorem}

\begin{proof}
1) Let $E=\left \{ \sigma (A/I),I\text{ faulty ideal of }A\right \} $.
Obviously $\sigma (A)\leq \sigma (A/I)$ for any faulty ideal $I$ of $A$,
thus $\sigma (A)$ is a lower bound for $E$. If it is proved that $\sigma
(A)\in E$, then we are done. By Lemma \ref{lemma11}, there exists an ideal $I$ of $A$
such that $A/I$ is coverable (then $I$ is faulty) and $\sigma (A)=\sigma
(A/I)$. This implies that $\sigma (A)\in E$.\newline
2) Let $E_{1}=\left \{ \sigma (A/\mathcal{P}_{1}\times ...\times A/\mathcal{P%
}_{\tau (f)}),f\in \mathfrak{F}(p),p\in \mathfrak{P}\right \} $. Obviously $%
\sigma (A)$ is a lower bound for $E_{1}$. We show that $\sigma (A)\in E_{1}$%
. Let $m=\sigma (A)$ and $A=\overset{m}{\underset{i=1}{\cup }}A_{i}$ be a
minimal covering of $A$. By Lemma \ref{lemma11}, let $I$ be an ideal of $A$ such that $%
A/I$ is finite, coverable and $\sigma (A/I)=\sigma (A)=m$, write $I$ in the
form $I=I_{1}...I_{w}$ where $I_{j}$ is a product of positive powers of
prime ideals of $A$ lying over the same prime number $p_{j}$. By Lemma \ref{lemma10}, $%
\sigma (A)=\sigma (A/I)=\underset{j=1}{\overset{w}{\inf }}\sigma (A/I_{j})$.
We may suppose that $\sigma (A)=\sigma (A/I_{1})$. Set $I_{1}=L_{1}...L_{z}$
where $L_{h}$ is a power of powers of prime ideals of $A$ lying over $p_{1}%
\mathbb{Z}$ and having the same residual degree $f_{h}$. By Theorem \ref{theorem25}, $%
\sigma (A/I_{1})=\underset{h=1}{\overset{z}{\inf }}\sigma (A/L_{h})$. We may
suppose that $\sigma (A/I_{1})=\sigma (A/L_{1})$. Set $L_{1}=\overset{z}{%
\underset{u=1}{\prod }}\mathcal{P}_{u}^{k_{u}}$, where $\mathcal{P}_{u}$ is
a prime ideal of $A$ lying over $p_{1}\mathbb{Z}$ of residual degree $f_{1}$%
. By Theorem \ref{theorem26}, $\sigma (A)=\sigma (A/L_{1})=\sigma (A/\mathcal{P}%
_{1}\times ...\times A/\mathcal{P}_{\tau (f_{1})})$, hence $\sigma (A)\in
E_{1}$.\newline
3) We omit the proof of this part since it is similar to that of 2).
\end{proof}

Let $K$ be a number field and $A$ be its ring of integers. We sumarize the
different possibilities for $\sigma (A)$.\newline
(1) If $A$ is not monogenic and $\mathfrak{P}=\varnothing $, then $A$ is
coverable but not finitely coverable, hence $\sigma (A)=\infty $.\newline
(2) If $A$ is not monogenic and $\mathfrak{P}\neq \varnothing $, then $A$ is
finitely coverable.\newline
(3) If $A$ is monogenic, generated by some unit, then $A$ is not coverable
and $\sigma (A)=\hat{\infty}$ by convention.\newline
(4) If $A$ is monogenic but no unit generates $A$, then $A$ is coverable but
not finitely coverable, thus $\sigma (A)=\infty $.

\begin{remark}\label{remark30}
Let $p\in \mathfrak{P}$. If $p$ is a common index divisor then $p<n$. If
above $p\mathbb{Z}$ there are at least $p$ prime ideals of $A$ with residual
degree over $p\mathbb{Z}$ equal to $1$, then $p$ divide $i(K)$ and $p\leq n$%
. This implies that $\mathfrak{P}$ is finite.\newline
Notice that these second kind of primes $p\in \mathfrak{P}$ satisfy the
condition $\lambda _{p}(1)\geq p$. If this inequality is strict, then $p$ is
a common index divisor.
\end{remark}

\begin{remark}\label{remark31}
Let $K$ be a number field, $A$ be its ring of integers and $p$ be a prime
number such that $p\in \mathfrak{P}$. Let $f\in \mathfrak{F}(p)$ and $%
\mathcal{P}_{1},...,\mathcal{P}_{\tau }$ be prime ideals of $A$ lying over $p%
\mathbb{Z}$ of residual degree equal to $f$, where $\tau =\tau (f)$. Let $%
R=A/(\mathcal{P}_{1}...\mathcal{P}_{\tau })=\underset{i=1}{\overset{m}{\cup }%
}S_{i}$ be a minimal covering of $R$. The value of $m$ is given in Theorem
\ref{theorem26}. For each $i=1,...,m$, let $A_{i}$ be a lift of $S_{i}$ in $A$, then $%
A_{i}$ is maximal and $A=\underset{i=1}{\overset{m}{\cup }}A_{i}$ is a
covering of $A$ but we may have $m>\sigma (A)$. However this covering of $A$
is irredundant. For, suppose that there exists $i\in \left\{ 1,...,m\right\} 
$ such that $A_{i}\subset \underset{j\neq i}{\cup }A_{j}$ and let $\Phi
:A\longrightarrow R=A/(\mathcal{P}_{1}...\mathcal{P}_{\tau })$ be the
canonical epimorphism, then $\Phi ^{-1}(S_{h})=A_{h}$ for $h=1,...,m$ and $%
S_{i}=\Phi (A_{i})\subset \underset{j\neq i}{\cup }\Phi (A_{j})=\underset{%
j\neq i}{\cup }S_{j}$, a contradiction.
\end{remark}

\textbf{Example \ \ }Let $K$ be a number field of degree $6$ and $A$ be its
ring of integers. Suppose that the splitting of $2A$ is given by $2A=%
\mathcal{P}_{1}\mathcal{P}_{1}^{\prime }\mathcal{P}_{2}\mathcal{P}%
_{2}^{\prime }$, where the suffix of each ideal is equal to its residual
degree. An instance of this is the field generated by a root of the
polynomial $g(x)=x^{6}+x^{5}+x^{4}-x^{3}+x^{2}+x+6$. With the notation of
Theorem \ref{theorem29}, we have $2\in \mathfrak{P}$. We look at the other candidates $%
p\in \mathfrak{P}$. Only $p=3$ and $p=5$ are candidates. Let $f$ be the
residual degree of some prime ideal of $A$ lying over $p\mathbb{Z}$ and let $%
\lambda _{p}(f)$ be the number of these ideals with residual degree equal to 
$f$. We must have $\lambda _{p}(f)\geq \tau (f)$.\newline
If $p=3$ and $f=2$, then $N_{3}(2)=3$ and $\tau (f)=4$.\newline
If $p=3$ and $f=3$, then $N_{3}(3)=8$ and $\tau (f)=9$.\newline
If $p=5$ and $f=2$, then $N_{5}(2)=10$ and $\tau (f)=11$.\newline
If $p=5$ and $f=2$, then $N_{5}(3)>6$ and $\tau (f)>7$.\newline
This shows that the only possibilities such that $3\in \mathfrak{P}$ is that 
$\lambda _{3}(1)\geq 3$. The similar condition for $5$ is $\lambda
_{5}(1)\geq 5$.\newline
For the specific field generated by a root of $g(x)$, we have $g(0)\equiv 0(%
\mod 3),g^{\prime }(0)\not\equiv 0$ $(\mod 3)$ and $%
g(-1)\not\equiv 0$ $(\mod 3)$, so that $\lambda _{3}(1)=1$.\newline
Similarly, we have\newline
\begin{equation*}
g(x)\equiv (x-1)^{2}(x^{4}+3x^{3}+X^{2}+3X+1)(\mod 5):\equiv
(x-1)^{2}h(x)(\mod 5)
\end{equation*}%
\newline
with $h(1)\not\equiv 0(\mod 3)$. Moreover we have\newline
\begin{equation*}
g(x)=(x-1)h(x)+5r(x)
\end{equation*}%
\newline
where $r(x)=x^{4}-x^{3}+x^{2}+1$. Since $r(1)\not\equiv 0(\mod 5)$,
then $5\nmid I(\theta )$. We have $h(0)\neq 0,h(-1)\neq 0$, hence $\lambda
_{5}(1)\leq 3$ and then $5\notin \mathfrak{P}$. We found that $\mathfrak{P=}%
\left\{ 2\right\} $ for the specific axample.\newline
By Theorem \ref{theorem29}, $\sigma (A)=\inf \left( \sigma (A/\mathcal{P}_{1}\mathcal{P}%
_{1}^{\prime }\right) ,\sigma (A/\mathcal{P}_{2}\mathcal{P}_{2}^{\prime
}))=\inf (3,4)=3$.\newline
We compute the covering of $A$ resulting as a lift of that of $R_{1}=A/\mathcal{%
P}_{1}\mathcal{P}_{1}^{\prime }\simeq A/\mathcal{P}_{1}\times A/\mathcal{P}%
_{1}^{\prime }$.\newline
The maximal subrings of $R_{1}$ are :
\begin{equation*}
S_{1}=\left\{ 0\right\} \times A/\mathcal{P}_{1}^{\prime },S_{2}=A/\mathcal{P%
}_{1}\times \left\{ 0\right\} \text{ and }S_{3}=\left\{ (a+\mathcal{P}_{1},a+%
\mathcal{P}_{1}^{\prime }),a\in A\right\} 
\end{equation*}%

hence $A=A_{1}\cup A_{2}\cup A_{3}$, where $A_{1},A_{2}$ and $A_{3}$ are
respectively $\mathcal{P}_{1},\mathcal{P}_{1}^{\prime }$ and $\mathcal{P}_{1}%
\mathcal{P}_{1}^{\prime }+\mathbb{Z}\left[ \alpha \right] $, $\alpha $ being
any element of $A$. We may take $\alpha =1$. This covering of $A$ holds with
the smallest number of components.\newline
We next look at the covering of $A$ deduced from a covering of $R_{2}=A/%
\mathcal{P}_{2}\mathcal{P}_{2}^{\prime }\simeq A/\mathcal{P}_{2}\times A/%
\mathcal{P}_{2}^{\prime }$.\newline
The maximal subrings of $R_{2}$ are :
\begin{eqnarray*}
T_{1} &=&\left\{ 0\right\} \times A/\mathcal{P}_{1}^{\prime },T_{2}=A/%
\mathcal{P}_{1}\times \left\{ 0\right\} ,T_{3}=\left\{ (a+\mathcal{P}_{2},a+%
\mathcal{P}_{2}^{\prime }),a\in A\right\}  \\
\text{and }T_{4} &=&\left\{ (a+\mathcal{P}_{2},a^{2}+\mathcal{P}_{2}^{\prime
}),a\in A\right\} =\left\{ (a+\mathcal{P}_{2},a+1+\mathcal{P}_{2}^{\prime
}),a\in A\right\} 
\end{eqnarray*}%

Their respective lifts in $A$ are given by
\begin{equation*}
B_{1}=\mathcal{P}_{2}+\mathbb{Z},B_{2}=\mathcal{P}_{2}^{\prime }+\mathbb{Z}%
,B_{3}=\mathcal{P}_{2}\mathcal{P}_{2}^{\prime }+\mathbb{Z}\left[ \mu \right] 
\text{ and }B_{4}=\mathcal{P}_{2}\mathcal{P}_{2}^{\prime }+\mathbb{Z}\left[
\gamma \right] 
\end{equation*}%

where $\mu $ is an element of $A$ satisfying the condition $\mu ^{2}+\mu
+1\equiv 0(\mod \mathcal{P}_{2})$ and $\gamma $ is defined as follows.%
\newline
Since $\mathcal{P}_{2}+\mathcal{P}_{2}^{\prime }=A$, there exist $u\in 
\mathcal{P}_{2}$ and $v\in \mathcal{P}_{2}^{\prime }$ such that $1=u+v$,
then we may take $\gamma =\theta u+(\theta +1)v$ where $\theta $ satisfies
the condition $\theta ^{2}+\theta +1\equiv 0(\mod \mathcal{P}_{2})$.
This element $\gamma $ satisfies the following congruences\newline
\begin{eqnarray*}
\gamma  &\equiv &(\theta +1)v(\mod \mathcal{P}_{2})\equiv (\theta
+1)(1-u)(\mod \mathcal{P}_{2})\equiv \theta +1(\mod \mathcal{P}%
_{2}) \\
\gamma  &\equiv &\theta u(\mod \mathcal{P}_{2}^{\prime })\equiv \theta
(1-v)(\mod \mathcal{P}_{2}^{\prime })\equiv \theta (\mod \mathcal{P%
}_{2}^{\prime })
\end{eqnarray*}%
\newline
These congruences imply
\begin{equation*}
\gamma +1\equiv \theta (\mod \mathcal{P}_{2})\text{ and }\gamma
+1\equiv \theta +1(\mod \mathcal{P}_{2}^{\prime })\text{.}
\end{equation*}

Notice that from the coverings $A=\underset{i=1}{\overset{3}{\cup }}A_{i}$
and $A=\underset{j=1}{\overset{4}{\cup }}B_{j}$, we may deduce a third one,
namely $A=\underset{(i,j)\in \left\{ 1,2,3\right\} \times \left\{
1,2,3,4\right\} }{\cup }A_{i}\cap B_{j}$. Obviously, for each $%
(i,j),A_{i}\cap B_{j}$ has finite index in $A$, but it is not maximal in $A$.


\newpage

\begin{thebibliography}{99}
\bibitem{A.K.} M. Ayad, O. Kihel, Common Divisors of values of Polynomials
and Common Factors of Indices in a Number Field, International J. Number
Theory 7, (2011), 1173-1194.

\bibitem{Ba} Bhargava, Groups as union of proper subgroups, Amer. Math.
Monthly, 116 (2009), 413-422.

\bibitem{B.F.S} R. A. Bryce. V. Fedri. L. Serena, Subgroup coverings of some linear groups, Bull. Aust. Math. Soc. 60 (1999) 227-238. 

\bibitem{En} Engstrom, On the common index divisors of an algebraic field,
Trans. A,M,S, 32 (1930), 223-237.

\bibitem{G.M.} G. Gunji, D. L. McQuillan, On a class of ideals in an
algebraic number field, J. Number Theory 2 (1970), 207-222.

\bibitem{Ha} H. Hancock, Foundations of the Theory of Algebraic Numbers,
Vol. 2, Dover Pub. (1964).

\bibitem{Ma} C. R. MacCluer, Common divisor of values of polynomials, J.
Number Theory, 3 (1971), 33-34.

\bibitem{Mu} F. F. Munoz, Subrings of Finite Commutative Rings, arxiv 06.12
(2017), 1-10.

\bibitem{Ne} B. Neumann. Groups covered by finitely many cosets, Publ. Math. Debrecen 3 (1954), 227-242.

\bibitem{Ri} P. Ribenboim, Classical Theory of Algebraic Numbers, Springer,
New York (2001).

\bibitem{Tm} M. J. Tomkinson, Groups covered by finitely many cosets or subgroups, Communications in Algebra, 15 (1987), 845-859.

\bibitem{We} N. J. Werner, Covering Numbers of Finite Rings, Amer. Math.
Monthly 122 (2015), 552-566.

\bibitem{Zy} E. Von Zylinski, Zur Theorie der aussewescentlicher
Discriminantenteiler algebraischer Korper, Math. Ann. 73, (1913), 273-274.
\end{thebibliography}
\end{document}